\documentclass{article}

\usepackage{amsmath}
\usepackage{amssymb}
\usepackage{amsthm}

\newtheorem{lemma}{Lemma}
\newtheorem{theorem}{Theorem}
\newtheorem{corollary}{Corollary}
\newtheorem{proposition}{Proposition}
\newtheorem{definition}{Definition}
\newtheorem*{proof of theorem 1}{Proof of Theorem $1$}
\newtheorem*{proof of theorem 3}{Proof of Theorem $3$}
\newtheorem*{proof of theorem 4}{Proof of Theorem $4$}
\newtheorem{remark}{Remark}
\newtheorem*{theorem a}{Theorem A}
\newtheorem*{corollary a}{corollary A}

\title{The local Morse Homology of the critical points in the Lagrange problem}
\author{Xiuting Tang\\fztxt@126.com}

\begin{document}

\maketitle

\begin{abstract}
In this paper, we construct local Morse homology in a new way and compute the local Morse homology of the critical points of the Lagrange problem. As a corollary, we prove for the first time that each of the linear critical points is either a saddle point or a degenerate critical point compared to the previous conclusion that if all the linear critical points are non-degenerate, they are saddle points.
\end{abstract}



\section{Introduction}





The Lagrange problem is the problem of two fixed centers adding an elastic force acting from the midpoint of the two masses. It was first observed by Lagrange \cite{Lagrange} that this problem is integrable. We refer to \cite{Hiltebeitel} for a comprehensive treatment which forces one can add to the problem of two fixed centers while still keeping the problem completely integrable. 
Setting the two fixed centers as $e=(-\frac{1}{2},0)$ and $m=(\frac{1}{2},0)$, the Hamiltonian of the Lagrange problem is

\begin{equation}H(q,p)=T(p)-U(q).\end{equation}\label{H of Lagrange}
where
$$T(p)=\frac{1}{2}|p|^2.$$
$$U:\mathbb{R}^2\backslash\{(\pm\frac{1}{2},0)\}\rightarrow \mathbb{R},
q\rightarrow \frac{m_1}{\sqrt{(q_1+\frac{1}{2})^2+q_2^2}}+\frac{m_2}{\sqrt{(q_1-\frac{1}{2})^2+q_2^2}}+\frac{\epsilon}{2}|q|^2,$$
$p=(p_1,p_2)^T, q=(q_1,q_2)^T$
and $m_1,m_2,\epsilon\in \mathbb{R}^+$.

If we set $\epsilon=1$ and $m_1=m_2=\frac{1}{2}$, then the Lagrange problem has the same potential energy as the circular restricted three body problem in the rotating coordinate, and the whole system is the circular restricted three body problem subtracting a Coriolis force. So the Lagrange problem reflect some information about the circular restricted three body problem near the boundary of the Hill's region.

In \cite{Tang}, we studied some properties of the Lagrange problem and found that it has five critical points $\{l_i| i=1, 2, 3,4,5\}$, where $l_4, l_5$ are maxima by direct computations.

In this paper we construct local Morse homology in a new way and compute the local Morse homology of these critical points. Compared to the conclusion in \cite{Tang} that if all the linear critical points are non-degenerate, they are saddle points, we prove for the first time that each of the linear critical points is either a saddle point or a degenerate critical point as a corollary.

\begin{theorem a}
For the collinear critical points $l_i(i=1, 2, 3)$ of the Lagrange problem with $m_1,m_2,\epsilon\in \mathbb{R}^+$, we have
$$HM^{loc}_*(l_i)=\left\{
\begin{aligned}
&\mathbb{Z}_2, *=1;\\
&\{0\}, otherwise.
\end{aligned}
\right.$$
For the maxima $l_i(i=4,5)$, we have
$$HM^{loc}_*(l_i)=\left\{
\begin{aligned}
&\mathbb{Z}_2, *=2;\\
&\{0\}, otherwise.
\end{aligned}
\right.$$
\end{theorem a}

\begin{corollary a}\label{saddle}
Each of the three critical points of the Lagrange problem with $m_1,m_2,\epsilon\in\mathbb{R}^+$ on the x-axis is either a saddle point or a degenerate critical point.
\end{corollary a}

This paper is organized as follows. In section 2, we construct local Morse homology completely in a new way. In section 2.1, we prove that the gradient flow lines behave locally. In section 2.2, we prove the weak compactness and Floer-Gromov convergence of the Moduli space of gradient flow lines similar as in the Morse homology. In section 2.3, we consider the transversality and prove that the moduli space of gradient flow lines is a manifold similar as in the Morse homolgy. In section 2.4, we give the well defined definition of local Morse homolgy. In section 2.5, we prove the invariance of local Morse homology with respect to the Morse-Smale pairs and small neighborhood we choose by considering the time dependent flow lines. In section 3, we give the local Morse homology of the critical points in the Lagrange problem and prove that each of the linear critical points is either a saddle point or a degenerate critical point for the first time.

\section{Local Morse homology}

\subsection{Introduction of  local Morse homology}

The constructions of local Morse homology and Floer homology go back to the original work of Floer(see, e.g.,\cite{Floer, Floer1}).
We can also refer to the constructions in \cite{Ginzburg} by Ginzburg.

There are some other researches about local Morse homology so far as we know. In \cite{Rot}, Rot and Vandervorst construct Morse-Conley-Floer homology using Conley's isolating blocks and Lyapunov functions. In \cite{Hein}, D. Hein, U. Hryniewicz, L. Macarini studied the transversality for local Morse homology with symmetries.

In \cite{Audin}, Mich\`ele Audin and Mihai Damian construct Morse homology from a modern point of view, and then naturally go to Floer homology. In this paper, we construct local Morse homolgy defined in the following from this modern point of view and then use it to study critical points in the Lagrange problem. 

We consider a smooth function $\tilde f$ on a Riemannian manifold $(M,g)$ with an isolated critical point $x_0$. Within a small neighborhood $B$ of the critical point, we perturb $\tilde f$ a little to get a Morse function $f$, in the progress of perturbation, there will be some critical points of $f$ appearing and disappearing in $B_\delta$. We then define the $k$-th local Morse Complex as the vector space of these critical points with Morse index k. There exits a gradient flow line of the equation
\begin{equation}
\partial_s x(s)=-\nabla f(x(s))
\end{equation}
connecting the $k$-th and $(k-1)$-th critical point.
Here we face the difficulty not exiting in the construction of Morse homology that we should insure the gradient flow lines behave locally, namely, it can not go out of a neighborhood $B$ of the critical point $x_0$. This can be proved by the energy analysis in Lemma \ref{lemma 2.2} in Section \ref{section 2.1}.

As in Morse homology, using the moduli space of unparametrised gradient flow lines
\begin{equation}\label{M}
\mathcal{M}(f, g; x^-, x^+):= \tilde{\mathcal{M}}(f, g; x^-, x^+)/\mathbb{R},
\end{equation}
we can define the boundary operator as
$$\partial: CM^{loc}_k(x_0,\delta, f, g)\rightarrow CM^{loc}_{k-1}(x_0, \delta, f, g),$$
$$\partial x^-=\sum_{\mu(x^-)=\mu(x^+)+1} \#_2\mathcal{M}(f, g; x^-, x^+) x^+.$$
Define local Morse homology as
$$HM^{loc}=\ker \partial/\text{Im} \partial.$$
To make this definition well defined, we should prove
$\partial^2=0.$ 
The strategy is to prove $\mathcal{M}(f, g; x^-, x^+)$ is a compact manifold for a generic choice of the Riemannian metric similarly as in the Morse homology.

The final step of the construction is 
to prove the invariance of local Morse homology with respect to Morse-Smale pairs $(f,g)$. It has much more difficulties not exiting in the construction of Morse homology. 
The core of the proof is to ensure that the time dependent gradient flow lines behave locally. Inspired by \cite{Frauenfelder}, we construct a specific homotopy of Morse functions and prove the local behavior by estimating the energy of time dependent gradient flow lines.

Finally, we construct local Morse homology which do not depend on the choice of the Morse-Smale paire and the small neighborhood we choose. 

In Theorem \ref{Homotopy invariance} in section $\ref{section 2.5}$, we prove the Homotopy invariance of local Morse homology, which is important in calculating the local Morse Homology of the critical points in the Lagrange problem.

\subsection{Local behavior of the gradient flow lines near critical points}\label{section 2.1}
$(M, g)$ is a Riemannian manifold,  $\tilde f: M \rightarrow \mathbb{R}$ is a smooth function on $M$. Let $x_0$ be an isolated critical point of $\tilde f$, i.e. $\nabla \tilde f(x_0)=0$. 


Define the neighborhood  $B_{\delta}=\{x\in M| ||x-x_0||_g
\leq \delta, \delta>0\}$ of $x_0$ and $$\Delta(\tilde f,x_0)=\{\delta| x_0 \text{ is the unique critical point of } \tilde f \text{ in } B_{2\delta}\}.$$
Choose $\delta \in \Delta$, by perturbing $\tilde f|_{B_{\delta}}$ a little bit, we can get a Morse function $f$ on $B_\delta$ for $f|_{M\backslash B_{\delta}}=\tilde f|_{M\backslash B_{\delta}}$ and the critical points of $f$ can only be born in $B_{\delta}$. Assume $x^-$ and $x^+$ are two critical points of $f$ and the difference of their Morse index $\mu$ is $1$, i.e, $\mu (x^-)-\mu(x^+)=1$, then there exits a gradient flow line of the equation
\begin{equation}\label{equ1}
\partial_s x(s)=-\nabla f(x(s))
\end{equation}
connecting $x^-$ and $x^+$, i.e. $x(s)\rightarrow x^{\pm}$ when $s\rightarrow \pm \infty$.

We claim in Lemma \ref{lemma 1.2} below that  by perturbing $\tilde f$ small enough, we can get the Morse function $f$ such that the gradient flow lines of equation (\ref{equ1}) can not go out of the neighborhood $B_{2\delta}$ of $x_0$.

Let $$\mathcal{F}=\{ f\in L^{\infty}(M, \mathbb{R})|  f| _{M\backslash B_\delta}=\tilde f|_{M\backslash B_\delta}\},$$
$$\mathcal{F}_{Morse}=\{f\in \mathcal{F}: f|_{B_\delta} \text{ is a Morse function}\}.$$
We know that $\mathcal{F}_{Morse}$ is dense in $\mathcal{F}$,
thus there exists a series of $f_\nu \in \mathcal{F}_{Morse}, \nu\in \mathbb{N}^*$, \begin{equation}\label{f lim}
f_\nu\overset{C^{\infty}}{\longrightarrow} \tilde f.
\end{equation}
In order to prove Lemma \ref{lemma 1.2}, we need Lemma \ref{lemma 1.1} below.
\begin{lemma}\label{lemma 1.1}
For any $\eta>0$, there exists $\nu_0$, such that when $\nu>\nu_0$, we have $\max f_\nu|_{C_\nu}-\min f_\nu|_{C_\nu}<\eta$, where $C_\nu=\{x\in B_\delta|\nabla f_\nu (x)=0\}$.
\end{lemma}
\begin{proof}

We assume the contradiction of the lemma, then there exists $\eta_0$,  such that for any $\nu\in \mathbb{N}$, $\max f_\nu|_{C_\nu}-\min f_\nu|_{C_\nu}\ge\eta_0$. Since $\#C_\nu<\infty$, there exist $x_\nu^1, x_\nu^2\in C_\nu, \nu\in \mathbb{N}^*$, such that $f_\nu (x_\nu^1)-f_\nu (x_\nu^2)\ge\eta_0$. Since $x_\nu^1, x_\nu^2$ are in the bounded regin $B_\delta$, there exist convergent subsequence $x^1_{\nu_j}\rightarrow x^1, x^2_{\nu_j}\rightarrow x^2, j\in \mathbb{N}^*$. As a result, $\tilde f(x^1)-\tilde f(x^2)\ge\eta_0$, this implies $\tilde f(x^1)\neq \tilde f(x^2)$ . By (\ref{f lim}), we know that $x^1$ and $x^2$ are critical points of $\tilde f$, i.e. $\tilde f(x_1)=\tilde f(x_2)=0$. This contradicts with the fact that $\tilde f$ has an unique critical point in $B_\delta$.
\end{proof}

\begin{lemma}\label{lemma 1.2}
For any $\delta\in\Delta$, there exists $f\in \mathcal{F}_{Morse}$ such that $x(s)\in B_{2\delta}, \forall s\in \mathbb{R}$.
\end{lemma}
\begin{proof}
Assume the contradiction, then for any Morse function $f(x)$, the flow should go out of $B_{2\delta}$ from $x^-$ and finally come back to $x^+$. Denoting $A_\delta=B_{2\delta}\backslash B_\delta$, then there should be at least twice the time that the flow go across $A_\delta$. Let the first time interval be $[t_0^-, t_0^+]$ and the last time interval $[t_1^-, t_1^+]$, then
\begin{equation}
\begin{aligned}
E(x)&=\int_{-\infty}^{+\infty}||\partial_s x(s)||^2ds \\
&\geq \int_{x_0^-}^{x_0^+}||\partial_s x(s)||^2ds+\int_{x_1^-}^{x_1^+}||\partial_s x(s)||^2ds.
\end{aligned}
\end{equation}
There exists $\epsilon>0$, such that $||\nabla f|_{A_\delta}||\ge \epsilon$. Since $\partial_s x(s)=-\nabla f|_{A_\delta}=-\nabla \tilde f|_{A_\delta}$, we have $||\partial_s x(s)||\ge \epsilon, t\in[t_0^-, t_0^+]\cup [t_1^-,t_1^+]$. As a result,

\begin{equation}\label{E1}
\begin{aligned}
E(x)&\geq  \epsilon \big( \int_{x_0^-}^{x_0^+}||\partial_s x(s)||ds+\int_{x_1^-}^{x_1^+}||\partial_s x(s)||ds \big)\\
&\geq 2\epsilon\delta.
\end{aligned}
\end{equation}
We can also calculate the energy $E(x)$ as follows.
\begin{equation}
\begin{aligned}
E(x)&=\int_{-\infty}^{+\infty}||\partial_s x(s)||^2ds\\
&=\int_{-\infty}^{+\infty}g(\partial_s x,\partial_s x)ds\\
&=\int_{-\infty}^{+\infty}g(\nabla f(x), \nabla f(x))ds\\
&=-\int_{-\infty}^{+\infty}\frac{d}{ds} f(x(s))ds\\
&=f(x^-)-f(x^+).
\end{aligned}
\end{equation}
Fix $\eta=\epsilon\delta$ in Lemma \ref{lemma 1.1}, we get a series of Morse functions $f_\nu (\nu>\nu_0)$ such that $\max{f_\nu}|_{C_\nu}-\min{f_\nu}|_{C_\nu}<\epsilon\delta$. Just choose a value of $\nu$ in $\nu>\nu_0$ and let $ f= f_\nu$, then

\begin{equation}\label{E2}
E(x)\leq \max f_\nu|_{C_\nu}-\min f_\nu|_{C_\nu}< \epsilon\delta.
\end{equation}
This contradicts (\ref{E1}).

\end{proof}

Define the set of Morse functions found in Lemma \ref{lemma 1.2} as $\mathcal{F}_{Morse}^{\delta}$ and

$$\mathcal{F}_{Morse}^{loc}=\bigcup_{\delta\in \Delta}\mathcal{F}_{Morse}^{\delta}.$$

\subsection{Compactness of the Moduli space of the gradient flow lines}

Let $x^+, x^-\in \text{crit} f$, $f\in \mathcal{F}_{Morse}^{\delta}$ and $g$ is a Riemannian metric on $B_{2\delta}$.
Define the Moduli space of parametrised gradient flow lines from $x^-$ to $x^+$ below.
\begin{equation}
\begin{aligned}
&\tilde{\mathcal{M}}(f, g; x^-, x^+)\\:=&\{x(s)\in C^\infty(\mathbb{R}, B_{2\delta})| \partial_s x(s)+\nabla_g f(x(s))=0, s\in\mathbb{R}, \lim_{s\rightarrow \pm\infty}=x^{\pm}\}
\end{aligned}
\end{equation}
There is a reparametrization action on $\tilde{\mathcal{M}}(f, g; x^-,x^+)$. Let $r_*x(s)=x(s+r), r\in \mathbb{R}$, if $x(s)\in \tilde{\mathcal{M}}(f, g; x^-,x^+)$, then $r_*x(s)\in \tilde{\mathcal{M}}(f, g; x^-, x^+)$.

If $x^-=x^+$, then $\tilde {\mathcal{M}}(f, g; x^-, x^+)=\{x\}$, $x$ is a constant gradient flow line and $r_*x=x, \forall r\in \mathbb{R}$.

If $x^-\neq x^+$, then there is a free $\mathbb{R}$ action on $ \tilde{\mathcal{M}}(f, g; x^-, x^+)$. We define the moduli space of unparametrised gradient flow lines below.
\begin{equation}\label{M}
\mathcal{M}(f, g; x^-, x^+):= \tilde{\mathcal{M}}(f, g; x^-, x^+)/\mathbb{R}.
\end{equation}
We want to prove in this section that $\mathcal{M}(f, g; x^-, x^+)$ is a compact set.

\begin{proposition}\label{prop of weak compactness}(Weak compactness)
Let $x_\nu\in C^{\infty}(\mathbb{R}, B_{2\delta})(\nu\in \mathbb{N})$ be solutions of equation (\ref{equ1}), then there exist a subsequence $\nu_j$ and gradient flow line $x$, such that $$x_{\nu_j}\overset{C_{loc}^\infty}{\longrightarrow}x.$$ i.e. for any $R>0$,
$$x_{\nu_j}|_{[-R, R]}\overset{C^\infty}{\longrightarrow}x|_{[-R, R]}.$$
\end{proposition}
\begin{proof}

Step 1:
Since $x(s)\in B_{2\delta}$ and $B_{2\delta}$ is compact, $||\nabla f(x_\nu)||_g$ is uniformly bounded. By (\ref{equ1}), $||\partial_s x_\nu (x)||_g$ is uniformly bounded, i.e. there exists $G>0$, such that $||\partial_s x_\nu(x)||_g\le G$. Then for any $\epsilon>0$, there exists ${\delta_0=\epsilon/G}$, when $|s_2-s_1|<\delta_0$, $$||x_\nu(s_2-s_1)||_g\leq ||x_\nu(s_\nu)||_g|s_2-s_1|\leq G|s_2-s_1|<G\delta_0\leq\epsilon,$$ where $s_1\leq s_\nu\leq s_2$, so $x_\nu$ is equicontinuous. By the theorem of Arzel\`{a}-Ascoli, there exist a subsequence $\nu_j$ and $x\in C^0(\mathbb{R}, M)$,
$$x_{\nu_j}\overset{C_{loc}^0}{\longrightarrow} x.$$

Step2:
Inductively.
When $x_{\nu_j}\overset{C_{loc}^0}{\longrightarrow} x$, we have
$$\nabla f(x_{\nu_j})\overset{C_{loc}^0}{\longrightarrow}\nabla f(x),$$ then $$\partial_s x_{\nu_j}(s)\rightarrow\partial_s x(s),$$ i.e. $$x_{\nu_j}\overset{C_{loc}^1}{\longrightarrow} x.$$
Let $s_0\in \mathbb{R}$, consider chart $V$ such that $x(s_0)\in V$, then there exists $\epsilon >0, j_0\in \mathbb{N}$, such that $x_{\nu_j}|_{[\epsilon_0-\epsilon,\epsilon_0+\epsilon]}\subset V, j\ge j_0$. Apply chain rule inductively to ODE $$\partial x_{\nu_j}(s)=-\nabla f(x_{\nu_j}(s)),$$
we get
$$\partial_s^k x_{\nu_j}=F_k(x_{\nu_j},\partial_s x_{\nu_j},...,\partial^{\nu-1}x_{\nu_j}),$$
$F_\nu$ is a continuous function with respect to $x_{\nu_j},\partial_s x_{\nu_j},...,\partial_s^{\nu-1}x_{\nu_j}$.
By $x_{\nu_j}\overset{C_{loc}^{\nu-1}}{\longrightarrow} x$, we get $x_{\nu_j}\overset{C_{loc}^\nu}{\longrightarrow}x$. Inductively, we get $x_{\nu_j}\overset{C_{loc}^\infty}{\longrightarrow}x$.
\end{proof}

\begin{definition}(Broken gradient flow line)
$x^{\pm}$ are critical points of the function $f$. A broken gradient flow line from $x^-$ to $x^+$ is a tupel
$$y=\{x^k\}_{1\le k \le n}, n\in \mathbb{N},$$
such that

(i) $x^k $ is a nonconstant gradient flow line, $1\le k \le n$.

(ii) $$\lim_{s\rightarrow -\infty} x^1(s)=x^-,$$
$$\lim_{s\rightarrow +\infty} x^k(s)=\lim_{s\rightarrow -\infty} x^{k+1}, 1\le k\le n-1,$$
$$\lim_{s\rightarrow +\infty} x^n(s)=x^+.$$
\end{definition}

\begin{definition}(Floer-Gromov convergence)
$x_\nu$ are a sequence of gradient flow lines, such that
$$\lim_{s\rightarrow +\infty} x_\nu(s)=x^{\pm}.$$
$y=\{x^\nu\}_{1\le k \le n}$ is a broken flow line from $x^-$ to $x^+$. $x_k$ is Floer-Gromov converges to $y$ if for $1\le k \le n$, there exists a sequence $\Gamma_{\nu}^k$, such that
$$(\Gamma_\nu^k)_* x_{\nu}\overset{C_{loc}^\infty}{\longrightarrow} x^k.$$
\end{definition}

\begin{theorem}\label{thm of Floer-Gromov compactness}(Floer-Gromov compactness)
If $f\in \mathcal{F}_{Morse}^{\delta}$ , $x_\nu$ is a sequence of gradient flow lines from $x^-$ to $x^+$, where $x^-$ is different from $x^+$. Then there exist $v_j$ and a broken gradient flow line $y=\{x^k\}_{1\le k\le n}$ from $x^-$ to $x^+$ such that $x_{\nu_j}$ Floer-Gromov converges to $y$.
\end{theorem}

If this theorem is true, then $\mathcal{M}(f, g; x^-, x^+)$ defined in (\ref{M}) is a compact set. In order to prove this theorem, we need Lemma \ref{lemma 2.1} and Lemma \ref{lemma 2.2} below.

\begin{lemma}\label{lemma 2.1}
If $f\in \mathcal{F}_{Morse}^{\delta}$, then $f$ has finite many critical points in $B_{2\delta}$.
\end{lemma}

\begin{proof}
Since $f|_{B_\delta}$ is Morse and only has critical points in $B_\delta$.  By Morse lemma, $f$ has finite many critical points in $B_\delta$, thus $f$ has finite many critical points in $B_{2\delta}$.
\end{proof}

\begin{lemma}\label{lemma 2.2}
If $f\in \mathcal{F}_{Morse}^{\delta}$, then there exist $x^\pm$ such that
$$\lim_{s\rightarrow \pm\infty}x(s)=x^\pm.$$
\end{lemma}

\begin{proof}
By Lemma \ref{lemma 2.1}, the number of critical points of $f$ is finite. i.e. $\#\text{crit} f<+\infty$. Let $\text{crit} f=\{x_1,...,x_N\}, N\in \mathcal {N}$. Choose neighborhoods $U_i$ of $x_i$ such that $U_i\cap U_j=\emptyset$.

Claim 1. There exists $\epsilon_0>0$ such that
$$||\nabla f(x)||_g\geq \epsilon_0,$$ for all $x\in M\backslash \cup_1^N U_i$.

Proof of Claim 1.
Since $B_{2\delta}$ is compact, $B_{2\delta}\backslash \cup_{i=1}^N U_i$ is compact. Assume by contradiction that $\epsilon_0$ does not exist, then for all $n\in\mathbb{N}$, there exist
$$y_n\in \backslash \cup_{i=1}^N U_i,$$
such that
\begin{equation}\label{gradient f}
||\nabla f(y_n)||_g\le \frac{1}{n}.
\end{equation}
$B_{2\delta}\backslash \cup_{i=1}^N U_i$ is compact implies that $y_n$ has converging subsequence $y_{n_j}$ such that
$$\lim_{j\rightarrow n}y_{n_j}=y\in B_{2\delta}\backslash \cup_{i=1}^N U_i.$$
By (\ref{gradient f}), $||\nabla f(y)||=0$, thus $y\in \text{crit} f$, contradicts $y\in B_{2\delta}\backslash U_{i=1}^N U_i$.

Let $0<\epsilon\le \epsilon_0$,
\begin{equation}\label{V_i^epsilon}
V_i^\epsilon:=\{x\in U_i| ||\nabla f(x)||_g<\epsilon\} \subset U_i.
\end{equation} is an open neighborhood of $x_i$.

Claim 2. For all $0< \epsilon \le \epsilon_0$, there exist $s_\nu^\epsilon\rightarrow +\infty$ such that
$$x(s_\nu^\epsilon)\in \cup_{i=1}^N V_i^\epsilon.$$

Proof of Claim 2.
Assume the contradiction, then there exists $\sigma\in\mathbb{R}$ such that $$x(s)\notin \cup_{i=1}^N V_i^\epsilon, \forall s\ge \sigma,$$
then $$||\nabla f(x(s))||_g\ge \epsilon, \forall \epsilon \ge\sigma.$$
Recall the energy
\begin{equation}
\begin{aligned}
E(x)&=\int_{-\infty}^{+\infty}||\partial_s x||_g^2 ds\\&=\int_{-\infty}^{+\infty}(-df(x(s))\partial_s x(s))ds\\&=\lim_{s\rightarrow -\infty} f(x(s))-\lim_{s\rightarrow +\infty} f(x(s))\\
&\le \max f-\min f.
\end{aligned}
\end{equation}
Then
\begin{equation}
\begin{aligned}
\max f-\min f&\ge E(x)\ge \int_{\sigma}^\infty||\partial_s x(s)||^2_g ds\\
&=\int_{\sigma}^\infty||\nabla f(x)||_g^2 ds\\
&\ge\int_{\sigma}^\infty \epsilon^2 ds\\
&=+\infty.
\end{aligned}
\end{equation}
This contradicts the fact that the energy should be finite.

Consider the compact subsets $\overline {V_i^{\frac{\epsilon}{2}}}$ and  $B_{2\delta}\backslash V_i^{\epsilon}$ of $B_{2\delta}$, $\overline {V_i^{\frac{\epsilon}{2}}}\cap M\backslash V_i^{\epsilon}=\emptyset$. Then the distance $d(\overline {V_i^{\frac{\epsilon}{2}}}, M\backslash V_i^{\epsilon})>0$.
Let $$K_\epsilon:=\min_{i=\{1,...,N\}} d(\overline {V_i^{\frac{\epsilon}{2}}}, M\backslash V_i^{\epsilon})>0.$$ Since
$$E=\int_{-\infty}^{\infty}||\partial_s x(s)||^2<+\infty,$$
there exists $\sigma_\epsilon\in\mathbb{R}$, such that
$$\int_{\sigma_\epsilon}^{+\infty}||\partial_s x(s)||_g ds\leq \frac{\epsilon K_\epsilon}{4}.$$

Claim 3. There exists $i\in\{1,...,N\}$ such that
$$x(s)\in V_i^\epsilon, s\ge \sigma_\epsilon.$$

Proof of Claim 3.
By claim 2, there exists $s_0\ge \sigma_\epsilon, i\in\{1,...,N\}$ such that $x(s_0)\in V_i^{\frac{\epsilon}{2}}$. Assume by contradiction, there exists $s_1\ge \sigma_\epsilon$ such that $x(s_1)\notin V_i^\epsilon$. Then there exists $t_0$ and $t_1$ such that $\sigma_\epsilon \le t_0\le t_1$, $x(t)\in V_i^\epsilon\backslash V_i^{\frac{\epsilon}{2}}$ for $t\in [t_0, t_1]$ and $x(t_0)\in\partial V_i^{\epsilon}, x(t_1)\in\partial V_i^{\frac{\epsilon}{2}}$ or $x(t_0)\in\partial V_i^{\frac{\epsilon}{2}}, x(t_1)\in\partial V_i^{\epsilon}$.
When $t\in[t_0, t_1]$, $$||\partial_s x(t)||_g=||\nabla f(x(t))||_g\ge \frac{\epsilon}{2}.$$
\begin{equation}
\begin{aligned}
d(x(t_0), x(t_1))&\le \int_{t_0}^{t_1}||\partial_s x(s)||_g ds\\
&\le \frac{2}{\epsilon}\int_{t_0}^{t_1}||\partial_s x(s)||_g ds\\
&\le\frac{2}{\epsilon}\int_{\sigma_\epsilon}^\infty||\partial_s x(t)||_g^2 ds\\
&\le \frac{2}{\epsilon}\cdot \frac{\epsilon K_\epsilon}{4}\\
&\frac{K_\epsilon}{2}.
\end{aligned}
\end{equation}
But this conflict with the fact that
$$d(x(t_0), x(t_1))\ge K_\epsilon\ge0.$$

Now we can prove this lemma by claim 3 directly.
By (\ref{V_i^epsilon}), for arbitrary small neighborhood of $x_i$, there exists $\epsilon>0$ such that $V_i^\epsilon\subset V$. By claim 3, there exists $i\in \{1,...,N\}$, for any $s\ge \sigma_\epsilon$, $x(s)\in V_i^\epsilon\subset V$. Thus $\lim\limits_{s\rightarrow +\infty }x(s)=x_i$.
Similary $\lim\limits_{s\rightarrow-\infty} x(s)$ exits.

\end{proof}

\begin{proof of theorem 1}
Let $m\in\mathbb{N}$,

$(A_m)$ There exist a subsequence $\nu_j$ and a broken gradient flow line $y_m=\{x^k\}_{1\le k\le l}$ with $l\le m$ such that there exist sequence $\Gamma_j^k\in \mathbb{R}$ for $1\le k\le l$, such that

$(i)$ $(\Gamma_j^k)_* x_{\nu_j}\overset{C_{loc}^\infty}{\longrightarrow} x^k, 1\le k\le l$.

$(ii)$ $\lim\limits_{s\rightarrow -\infty} x^1(s)=x^-$.

$(iii)$ If  $l<m$, $\lim\limits_{s\rightarrow +\infty} x^l(s)=x^+$.

Proof of $(A_1)$. By the proof of Lemma \ref{lemma 2.1}, the critical points of $f$ are isolated. Choose open neighborhood $V$ of $x^-$ such that $\overline V \cap \text{crit} f=\{x^-\}$.
Define the first exit time of $x(s)$ from $V$ as
$$\Gamma^1_\nu:=inf\{s\in\mathbb{R}, x_\nu (s)\notin V\}$$
$\Gamma_\nu^1$ is finite since $\lim\limits_{s\rightarrow-\infty} x_\nu(s)=x^-$ and $\lim\limits_{s\rightarrow+\infty} x_\nu(s)=x^+$.
By Proposition \ref{prop of weak compactness}, there exist a subsequence $\nu_j$ and a gradient flow line $x^1$ such that
$$(\Gamma_{\nu_j}^1)_* x_{\nu_j}\overset{C_{loc}^\infty}{\longrightarrow}x^1.$$
$x^1$ is nonconstant since
$$x^1(0)=\lim_{j\rightarrow+\infty}(\Gamma_{\nu_j}^1)_* x_{\nu_j}=\lim_{j\rightarrow+\infty} x_{\nu_j}(r_{\nu_j}^1)\in\partial V.$$
and
$$\partial V\cap \text{crit} f=\emptyset.$$
Define $y^1=\{x^1\}$, $\Gamma_j^1=\Gamma_{\nu_j}^1$, $(i)$ is proved.
To show $(ii)$, by Lemma \ref{lemma 2.2}, there exists $x\in\text{x}$ such that $\lim\limits_{x\rightarrow-\infty} x^1(s)=x$.
By definition of $\Gamma_\nu^1$, $(\Gamma_{\nu_j}^1)_* x_{\nu_j}(s)\in\overline{V}, s\le 0$, then $\lim\limits_{s\rightarrow -\infty} x^1(s)\in \text{crit} f\cap \overline{V}=\{x^-\}$. $(ii)$ is proved and (iii) is empty.

Now assume $(A_m)$ is true, we want to prove $(A_{m+1})$.
Let $y_m=\{x^k\}_{1\le k\le l}$ be broken flow line constructed in $(A_m)$.

Case $1$. $\lim\limits_{s\rightarrow+\infty}x^l(s)=x^+$.
Put $y_{m+1}=y_m$, then $(A_{m+1})$ is satisfied.

Case $2$. Asumme $\lim\limits_{s\rightarrow+\infty} x^l(s)\neq x^+$, then $l=m$. By Lemma \ref{lemma 2.2}, there exists $(x^m)^+\in\text{crit}f$ such that $$\lim_{s\rightarrow+\infty} x^m(s)=(x^m)^+.$$
Choose open neighborhood $V$ of $(x^m)^+$ such that $$\overline{V}\cap \text{crit} f=\{(x^m)^+\}.$$
By $\lim\limits_{s\rightarrow+\infty} x^m(s)=(x^m)^+$, there exists $s_0\in\mathbb{R}$ such that $$x^m(s)\in V, \forall s\ge s_0$$ and $$(\Gamma_j^m)_* x_{\nu_j}\overset{C_{loc}^\infty}{\longrightarrow}x^m.$$ Then there exists $j_0\in\mathbb{N}$ such that $$(\Gamma_j^m)_* x_{\nu_j}(s_0)\in V, j\ge j_0.$$
Define
\begin{equation}\label{def of Rj}
R_j:=\inf\{r\ge0:(\Gamma_j^m)_*x_{\nu_j}(s_0+r)\notin V\}.\end{equation}
Properties of $R_j$:

$(i)$ Since $\lim\limits_{s\rightarrow+\infty(\Gamma_j^m)_* x_{v_j}(s)}=x^+\notin\overline{V}$, we have $R_j<+\infty, j\ge j_0$.

$(ii)$ Since $x^m(s)\in V$ for $s\ge s_0$ and $(\Gamma_j^m)_* x_{v_j}(s)\overset{C_{loc}^\infty}{\longrightarrow}x^m(s)$, we have $\lim\limits_{j\rightarrow+\infty} R_j=+\infty$.
Define $$\Gamma_j^{m+1}:=\Gamma_j^m+s_0+R_j.$$
By Proposition \ref{prop of weak compactness}, we have
$$(\Gamma_j^{m+1})_* x_{v_j}\overset{C_{loc}^\infty}{\longrightarrow}x^{m+1}.$$

Claim $1$. $x^{m+1}$ is nonconstant.

Proof of Claim $1$.

$(\Gamma_j^{m+1})_* x_\nu(0)\in \partial V$ implies that $x^{m+1}(0)\in \partial V$. By $\partial V\cap \text{crit} f=\emptyset$, we get $x^{m+1}$ is nonconstant.

Claim $2$. $\lim\limits_{s\rightarrow+\infty}x^{m+1}(s)=(x^m)^+$.

Proof of Claim $2$.

By Lemma \ref{lemma 2.2}, there exists $(x^{m+1})^-\in \text{crit} f$ such that $\lim\limits_{s\rightarrow-\infty}x^{m+1}(s)=(x^{m+1})^-$. Now we show that $(x^{m+1})^-=(x^m)^+$ by contradiction. Assume $(x^{m+1})^-\neq (x^m)^+$. Since $\overline{V}\cap \text{crit} f=\{(x^m)^+\}$, we can choose an open neighorhood $V'$ of $(x^{m+1})^-$ such that
\begin{equation}\label{a}
\overline{V}\cap \overline{V'}=\emptyset.
\end{equation}
$\lim\limits_{s\rightarrow-\infty} x^{m+1}(s)=(x^{m+1})^-$ implies that there exists $s_1<0$ such that $$x^{m+1}(s)\in V', \forall s\le s_1.$$
$(\Gamma_j^{m+1})_* x_{v_j}\overset{C_{loc}^\infty}{\longrightarrow}x^{m+1}$ implies that there exists $j_0'$ such that $\forall j\ge j_0'$,
\begin{equation}\label{b}
(\Gamma_j^{m+1})_*x_{v_j}(s_1)\in V'.
\end{equation}
Definition (\ref{def of Rj}) of $R_j$ implies that
\begin{equation}\label{c}
(\Gamma_j^m)_* x_{v_j}(s)\in V, s\in(s_0, s_0+R_j).
\end{equation}
By $$(\Gamma_j^m)_* x_\nu(s_0+R_j)=x_v(\Gamma_j^m+s_0+R_j)=(\Gamma_j^{m+1})_* x_\nu(0),$$ we have
\begin{equation}
\begin{aligned}
(\Gamma_j^m)_* x_{v_j}(s)&=x_{v_j}(s)\\
&=x_{v_j}(s+\Gamma_j^m)\\
&=x_{v_j}(s+\Gamma_j^{m+1}-s_0-R_j)\\
&=(\Gamma_j^{m+1})_* x_{v_j}(s-s_0-R_j).
\end{aligned}
\end{equation}
Denoting $\theta=s-s_0-R_j$, by (\ref{c}), we have
\begin{equation}\label{d}
(\Gamma_j^{m+1})_*x_{v_j}(\theta)\in V, \theta\in (-R_j, 0).
\end{equation}
Since $\lim_{j\rightarrow +\infty} R_j=+\infty$, there exists $j$ such that $-R_j\le s_1$.
By (\ref{b}), (\ref{d}) and then (\ref{a}), we have
$$(\Gamma_j^{m+1})_* x_{v_j}(s_1)\in V'\cap V=\emptyset.$$

Define $y_{m+1}=\{x^k\}_{1\le k\le m+1}$, then $y_{m+1}$ is a broken gradient flow line and satisfies $(A_{m+1})$.

Since $(x^m)^+\neq(x^{m'})^-$ for $m\neq m'$ and $\#\text{crit}f <+\infty$, the induction stabilizes after at most $\#\text{crit} f$ many steps. i.e. There exists $m\in\mathbb{N}$ such that $$(x^m)^+=x^+.$$
\end{proof of theorem 1}

\subsection{Transversality and the moduli space of the gradient flow lines as a manifold}


We define the Hilbert manifold $\mathcal{H}=\mathcal{H}_{x_-, x_+}$ of $W_{1,2}$-trajectories from $x^-$ to $x^+$. We define charts below.

Let $x\in C^\infty (\mathbb{R}, B_{2\delta})$ such that $\exists T>0$,
$$x(s)=x^-, s\ge T,$$
$$x(s)=x^+, s\le T.$$
$x^* TB_{2\delta}$ is the vector bundle over $\mathbb{R}$.
Choose an open neighborhood $V_x$ of zero section such that $\forall r\in \mathbb{R}$, the restriction of the exponential map
$$exp: T_{x(r)} B_{2\delta}\cap V_x\rightarrow B_{2\delta}$$
is injective.

Choose a trivialization $\Phi$ of the vector bundle $x^* TB_{2\delta}$ over $\mathbb{R}$, $$\Phi: x^*TB_{2\delta}\rightarrow \mathbb{R}\times\mathbb{R}^n$$
Denoting $$U_x=\{\xi\in W^{1,2}(\mathbb{R}, \mathbb{R}^n)| (r, \xi(r)\in\Phi(V_x), \forall r\in \mathbb{R})\}\in W^{1,2}(\mathbb{R}, \mathbb{R}^n).$$
We then have charts
$$\phi_x: U_x\rightarrow \mathcal{H}, \xi\mapsto exp(\Phi^{-1}(\xi)).$$


Define a $L^2$-bundle over $\mathcal{H}$ $$s: \mathcal{H}\rightarrow \mathcal{E}$$
The fiber on $x\in\mathcal{H}$ is $$\mathcal{E}_x:=L^2(x^*TM).$$
$s_x=\pi^{-1}(x), \pi\circ s=id_\mathcal{H}$. We identify $\mathcal{H}$ with the zero section of $\mathcal{E}$.
Define a section that depend on the Riemannian metric $g$
$$s_g:\mathcal{H}\rightarrow \mathcal{E}, s_g(x)=\partial_s x+\nabla_g f(x).$$
Then the zero set of the section $s_g$ is the moduli space of parametrized gradient flow lines from $x_-$ to $x^+$. i.e.
$$s_g^{-1}(0)=\tilde{M}(f, g; x^-, x^+)$$

Now we define the vertical differential of section $s$. If $x\in\mathcal{H}$, then the differential of the section is a linear map
$$ds(x): T_x\mathcal{H}\rightarrow T_{s(x)}\mathcal{E}.$$
Note that if $x\in\mathcal{H}$ we have a canonical splitting of the tangent space $T_x\mathcal{E}$ into horizontal and vertical subspaces
$$T_x \mathcal{E}=T_x\mathcal{H}\oplus\mathcal{E}_x.$$
Let $\pi:T_x\mathcal{E}\rightarrow \mathcal{E}_x$ denote the projection along $T_x\mathcal{H}$ and define the vertical differential
$$Ds(x):=\pi\circ ds(s): T_x\mathcal{H}\rightarrow \mathcal{E}_x.$$

\begin{definition}
A section $s:\mathcal{H}\rightarrow\mathcal{E}$ is called transverse to the zero section, denoted by $s\pitchfork 0$, if $D_s(x)$ is surjective $\forall x\in s^{-1}(0)$.
\end{definition}

\begin{definition}
 A section $s:\mathcal{H}\rightarrow\mathcal{E}$ is called a Fredholm section if $$Ds(x):T_x\mathcal{H}\rightarrow\mathcal{E}_x$$ is a Fredholm operator $\forall x\in s^{-1}(0)$.
\end{definition}

The implicit function theorem tells us the following theorem, refer to Appendix A in \cite{McDuff}.
\begin{theorem}
$s:\mathcal{H}\rightarrow\mathcal{E}$ is a transverse Fredholm section, then $s^{-1}(0)$ is a manifold, $x\in s^{-1}(0)$.
$$\dim_x s^{-1}(0)=\dim(\ker D_s(x))=\text{ind} D_s(x).$$
\end{theorem}

\begin{corollary}\label{s^-1 manifold}
If $s\pitchfork 0$, then $s^{-1}(0)=\tilde M(f, g; x^-, x^+)$ is a manifold and
$$dim \tilde{M}(f, g; x^-, x^+)=\mu(x^-)-\mu(x^+).$$
\end{corollary}

Let $\mathcal{M}$ be the space of all Riemannian metrics on $B_{2\delta}$ endowed with $C^\infty$-topology.
$\mathcal{U}\subset \mathcal{M}$ is called second category if there exist open dense subsets  $\mathcal{M}_i\subset \mathcal{M}$ for $1\le i\le +\infty$ such that $\mathcal{U}=\cap_{i=1}^{+\infty} \mathcal{M}_i$.
By Baire theorem, a set of second category is dense.
\begin{theorem}(Smooth Transversality)\label{thm of smooth transversality}
There exists $\mathcal{U}\subset \mathcal{M}$ of second category such that $$s_g\pitchfork 0, \forall g\in \mathcal{U}.$$
i.e. $\mathcal{U}$ is countable intersection of open and close subsets.
\end{theorem}

Let $\mathcal{M}^k$ be the space of Riemannian metrics on $B_{2\delta}$ of regularity $C^k$. Note that $\mathcal{M}$ is not a Banach manifold but just a Fr\'echet manifold. $\mathcal{M}^k$ is a Banach maninfold and we can apply Sard's theorem directly. We can obtain the Smooth Transversality Theorem from the $C^k$ Transversality Theorem below for all large enough integers $k$ due to an argument of Taubes, refer to section 7.7 in \cite{McDuff}.

\begin{theorem}( $C^k$ Transversality)\label{thm of C^k transversality}
There exist $\mathcal{U}\subset \mathcal{M}^k$ of second category such that
$$s_g\pitchfork 0,\forall g\in\mathcal{U}.$$
\end{theorem}

Consider the section
$$S: \mathcal{M}^k\times \mathcal{H}\rightarrow\mathcal{E}.$$
If $(g, x)\in S^{-1}(0)$, then $x$ is a gradient flow line from $x^-$ to $x^+$ with respect to $\nabla_g f$.

\begin{proposition}\label{prop of DSg}
Assume $g\in\mathcal{M}^k, x\in s^{-1}_g(0)$, then
$$DS (g, x): T_g\mathcal{M}\times T_x\mathcal{H}\rightarrow \mathcal{E}_x$$ is surjective.
\end{proposition}
\begin{proof}
Case 1. $x^-\neq x^+$.

$(g, x)\in S^{-1}(0)$. i.e. $\partial_s x+\nabla_g f(x)=0$.
For Riemannian metric $g_1$, $g_2$ and $\lambda_1, \lambda_2\in \mathbb{R}^*$, $\lambda_1g_1+\lambda_2g_2$ is again a Riemannian metric.
Define the linear map for $y\in B_{2\delta}$
$$L_{y,g}: Sym(T_y B_{2\delta})\rightarrow T_y B_{2\delta}, h\mapsto \frac{d}{dr}\bigg|_{r=0}\nabla_{g+rh}f(y)$$

The vertical differential of $S$ at $(g, x)$ is
\begin{equation}\label{DS}
DS(g, x): T_g\mathcal{M}^k\oplus T_x\mathcal{H}\rightarrow \mathcal{E}_x, (h,\xi)\mapsto Ds_g(x)\xi+L_{x,g}h.
\end{equation}

On $\mathcal{E}_x$ we define the $L^2$ inner product, for $\eta_1, \eta_2$
$$\langle\eta_1,\eta_2\rangle_g=\int_{-\infty}^{+\infty}g_{x(s)}(\eta_1(s),\eta_2(s))ds, \eta_1,\eta_2\in\mathcal{E}_x.$$
Let
\begin{equation}\label{eta}
\eta\in \text{im} DS(g, x)^\perp=\text{coker}(DS(g, x)),
\end{equation}
we want to show that $\eta$ vanishes. (\ref{eta}) is equivalent to
$$\langle DS(g, x)(h,\xi), \eta\rangle_g=0, \forall (h,\xi)\in T_g\mathcal{M}^k\oplus T_x\mathcal{H}.$$
By (\ref{DS}), this implies
\begin{equation}\label{Ds,L1}
\langle Ds_g(x)\xi, \eta \rangle_g, \forall \xi\in T_x\mathcal{H}
\end{equation}
and
\begin{equation}\label{Ds,L2}
\langle L_{x,g}h,\eta\rangle_g, \forall h\in T_g\mathcal{M}^k.
\end{equation}

(\ref{Ds,L1}) implies that $\eta$ lies in the cokernel of $Ds_g(x)$, $
\eta\in Ds_g^\perp$, i.e. $$\partial_s\xi+A(s)\xi=0.$$ Then we have $$\partial_s\eta-A^T(s)\eta=0,$$i.e. $\eta$ lies in the kernel of the adjoint of $Ds_g(x)$. As a result, $\eta$ is continuous.

Note that \begin{equation}\label{g,df}
g(\nabla_gf(x),\eta)=df(x)\eta.
\end{equation}
Differentiating with respect to the metric $g$ we get
$$\frac{d}{dr}\bigg|_{r=0}(g+rh)(\nabla_{g+rh}f(x),\eta)=h(\nabla_gf(x),\eta)+g(L_{x,g}h,\eta).$$
Since the right handside of $(\ref{g,df})$ is independent of $g$, we get
$$h(\nabla_gf(x),\eta)+g(L_{x,g}h,\eta)=0.$$
Together with (\ref{Ds,L2}), we get

\begin{equation}\label{integral of h}
0=\langle L_{x,g}h,\eta\rangle_g=\int_{-\infty}^{+\infty}g(L_{x,g}h,\eta)ds=-\int_{-\infty}^{+\infty}h(\nabla_gf(x),\eta)ds.
\end{equation}

Assume by contradiction $\exists s_0 \in\mathbb{R}$ such that $\eta(s_0)\neq 0$.
Since $x^-\neq x^+$, $x$ is a nonconstant gradient flow line flow downhill, we have $\nabla_g f(x(s))\neq 0, \forall s\in \mathcal{R}$. By Lemma \ref{Auxiliary lemma} below, there exists symmetric bilinear form $ h_0$ on $T_{x(s_0)} M$ such that $$h_0(\nabla_g f(x(s_0), \eta{s_0}))>0.$$
This contradict with (\ref{integral of h}), as a result $\eta\equiv 0$.

Case 2. $x^-\neq x^+$.
$\forall g\in \mathcal{M}$, $s^{-1}_g(0)=x(s)=x^-=x^+$, where $x(s)$ is constant.

Claim. $Ds_g(x)$ is surjective for all $g\in \mathcal{M}^k$.
The linearisation of the gradient flow equation of (\ref{equ1}) is given by
\begin{equation}\label{DA}
D_A\xi=\partial_s\xi+A\xi,
\end{equation}
where $A$ is the Hessian of $f$  and $x^-=x^+$. Since $f$ is Morse after changing coordinates we can assume that
$$A=\left(\begin{matrix} a_1 & \cdot\cdot\cdot &0\\ & \cdot\cdot\cdot &  \\0 & \cdot\cdot\cdot  &a_n
\end{matrix}\right), a_i\neq 0, 1\le i\le n.$$
Assume $\xi\in \text{ker}D_A$, $\xi=(\xi_1,...,\xi_n)$, then (\ref{DA}) implies that
$$\xi_i(s)=\xi_i(0)e^{-a_is}, 1\le i\le n.$$
Since $\xi\in W^{1,2}$, we have $\xi\equiv 0$ and ker$D_A=\{0\}$.
Therefore, $$\text{im}D_A^\perp=\text{ker}D_{-A}=\{0\}$$.

\end{proof}

\begin{lemma}\label{Auxiliary lemma}
Assume $v, w\in \mathbb{R}\backslash\{0\}$, then there exists symmetric bilinear form $h$ such that $h(v, w)>0$.
\end{lemma}
\begin{proof}
Refer to Lemma 7.8 in \cite{McDuff}.
\end{proof}

By Proposition \ref{prop of DSg} and the implicit function theorem we have that $S^{-1}(0)$ is a Banach manifold. Consider the map
\begin{equation}\label{Pi}
\Pi: S^{-1}(0)\rightarrow \mathcal{M}^k, (g, x)\rightarrow g.
\end{equation}
If the differential
\begin{equation}\label{dPi}
d\Pi (y):T_yS^{-1}(0)\rightarrow T_g\mathcal{M}^k
\end{equation}
is surjective, then $g$ is called a regular value of $\Phi$ for all $y\in\Phi^{-1}(g)$. Since $S^{-1}(0)$ is a Banach manifold, by Sard's theorem
$$\mathcal{M}^k_{reg}=\{g\in \mathcal{M}^k: g \text{ is a regular value of } \Pi \}.$$
is of second category in $\mathcal{M}$.

\begin{lemma}\label{Dsg is surjective}
If $g\in \mathcal{M}_{reg}^k$, then $Ds_g$ is surjective for every $(g, x)\in\Pi^{-1}(g)$.
\end{lemma}

\begin{proof}
$y=(g,x)\in S^{-1}(0)$,
$$T_{(g,x)}S^{-1}(0)=\{(h,\xi)\in T_g\mathcal{M}\times T_x{\mathcal{H}}\big| DS(g, x)(h, \xi)=0\}.$$
By (\ref{DS}),

\begin{equation}\label{TS^-1(0)}
T_{(g,x)}S^{-1}(0)=\{(h,\xi)\in T_g\mathcal{M}\times T_x{\mathcal{H}}\big| Ds_g(x)\xi=-L_{x,g}h\}.
\end{equation}

Choose $\eta\in\xi_x$.
Since $DS(g, x):T_g{\mathcal{M}}\times T_x\mathcal{H}\rightarrow\mathcal{E}_x$ is surjective, $\exists (h_0,\xi_0)\in T_g\mathcal{M}\times T_x\mathcal{H}$ such that
\begin{equation}\label{1}
DS(g, x)(h_0,\xi_0)=\eta.
\end{equation}
For $g\in \mathcal{M}_{reg}^k$, $$d\Pi(g,x): T_{g,x}S^{-1}(0)\rightarrow T_g\mathcal{M}^k$$
is surjective, then there exists $(h_1,\xi_1)\in T_g\mathcal{M}\times T_x\mathcal{H}$ such that $$d\Pi(g,x)(h_1,\xi_1)=h_0.$$ In fact, $h_1$ is just $h_0$ by definition (\ref{Pi}) of $\Pi$. i.e.
\begin{equation}\label{4}
h_1=h_0
\end{equation}
Since $(h_1,\xi_1)\in T_g\mathcal{M}\times T_x\mathcal{H}$, by (\ref{TS^-1(0)})
\begin{equation}\label{3}
D_{s_g}(x)\xi_1=-L_{x,g}h_1.
\end{equation}
By (\ref{1}), (\ref{DS})(\ref{4}) and (\ref{3}),
\begin{equation}
\begin{aligned}
\eta&=DS(g,x)(h_0, \xi_0)\\
&=D_{s_g}(x)\xi_0+L_{x,g}h_0\\
&=D_{s_g}(x)\xi_0+L_{x, g} h_1\\
&=D_{s_g}(x)\xi_0+D_{s_g}(x)\xi_1\\
&=D_{s_g}(x)(\xi_0-\xi_1)
\end{aligned}
\end{equation}
Setting $\xi=\xi_0-\xi_1\in T_x\mathcal{H}$, this becomes
$$Ds_g(x)\xi=\eta,$$
i.e. $D_{s_g}$ is surjective.

\end{proof}

\begin{proof of theorem 4}
By Lemma \ref{Dsg is surjective} and Sard's theorem we get this theorem directly.
\end{proof of theorem 4}

\begin{proof of theorem 3}
By Theorem \ref{thm of C^k transversality} and argumemt of Taubes, we get this theorem.
\end{proof of theorem 3}

\subsection{Definition of local Morse homology}

Define the local Morse complex as the vector space
$$CM^{loc}_k(f)=\{\sum_{c\in \text{Crit}_k(f)}a_c c|a_c\in \mathbb{Z}/2\},$$
where $\text{Crit}_k(f)$ denotes the set of critical points of Morse index of $k$.

\begin{definition}
A tuple $(f,g)$ consisting of $f\in\mathcal{F}_{Morse}^{\delta}$ and $g\in \mathcal{M}_{reg}(f)$ is called a local Morse-Smale pair.
\end{definition}
If $(f,g)$ is a local Morse-Smale pair, then by Corollary \ref{s^-1 manifold} and Theorem \ref{thm of smooth transversality} we know that $\tilde{\mathcal{M}}(f, g; x^-, x^+)$ is a mainfold and $$\text{dim}\tilde{M}(f, g; x^-, x^+)=\mu(x^-)-\mu(x^+).$$
Then $$\text{dim}\mathcal{M}(f, g; x^-, x^+)=\text{dim}\tilde{\mathcal{M}}(f, g; x^-, x^+)/\mathbb{R}=\mu(x^-)-\mu(x^+)-1.$$

If $\mu(x^-)=\mu(x^+)+1$ for $x^-, x^+\in \text{crit} f$, together with Theorem \ref{thm of Floer-Gromov compactness},
$\mathcal{M}(f, g; x^-, x^+)$ is a compact manifold of dimension $0$, then $\mathcal{M}(f, g; x^-, x^+)$ consist of finite points.

Define boundary operator
$$\partial: CM^{loc}_k(x_0,\delta, f, g)\rightarrow CM^{loc}_{k-1}(x_0, \delta, f, g)$$
$$\partial x^-=\sum_{\mu(x^-)=\mu(x^+)+1} \#_2\mathcal{M}(f, g; x^-, x^+) x^+.$$

If $\mu(x^-)=\mu(x^+)+2$ for $x^-, x^+\in \text{crit}f$, then $\mathcal{M}(f, g; x^-, x^+)$ is a compact manifold with dimension 1. Note that the only examples of compact $1$-dim manifold are finite disjoint unions of closed interval and circles, therefore if $X$ is a $1$-dim compact manifold, then $\# \partial_2 X$ is even.
\begin{equation}
\begin{aligned}
&\partial^2 x^-\\=&\partial \big(\sum_{c\in \text{crit}f,  \mu(x^-)=\mu(c)+1}\#_2\mathcal{M}(f, g; x^-, c)c \big)\\
=&\sum_{c\in \text{crit}f,  \mu(x^-)=\mu(c)+1}\#_2\mathcal{M}(f, g; x^-, c)\partial c\\
=&\sum_{c\in \text{crit}f,  \mu(x^-)=\mu(c)+1}\#_2\mathcal{M}(f, g; x^-, c) \sum_{\mu(x^-)=\mu(x^+)+2}\#_2\mathcal{M}(f, g; c, x^+)x^+\\
=&\sum_{\mu(x^-)=\mu(x^+)+2} \big(\sum_{c\in\text{crit}f, \mu(x^-)=\mu(c)+1} \#_2\mathcal{M}(f, g; x^-, c) \#_2\mathcal{M}(f, g; c, x^+) \big) x^+\\
=&\sum_{\mu(x^-)=\mu(x^+)+2} \bigg(\#_2\big(\bigcup_{c\in\text{crit}f, \mu(x^-)=\mu(c)+1}\mathcal{M}(f, g; x^-,c)\times\mathcal{M}(f, g; c, x^+) \big) \bigg).
\end{aligned}
\end{equation}
Where $\bigcup_{c\in\text{crit}f, \mu(x^-)=\mu(c)+1}\mathcal{M}(f, g; x^-,c)\times\mathcal{M}(f, g; c, x^+) $ is the boundary of the manifold $\mathcal{M}(f, g; x^-, x^+)$,by the disussion above, its moduli $2$ sum vanishes, i.e.
$$\#_2\big( \bigcup_{c\in\text{crit}f, \mu(x^-)=\mu(c)+1}\mathcal{M}(f, g; x^-,c)\times\mathcal{M}(f, g; c, x^+) \big)=0.$$
As a result, $$\partial^2=0.$$

Therefore, if $(f, g)$ is a local Morse smale pair, we can define local Morse homology in $B_{2\delta}$. $$HM^{loc}_k(\tilde{f}, x_0, \delta, f, g)=\ker \partial_k(\tilde{f}, x_0, \delta, f,g)/\text{im }\partial_{k+1}(\tilde{f}, x_0, \delta, f, g).$$


\subsection{Invariance of local Morse homology }\label{section 2.5}

We aim to prove that $HM^{loc}_*$ does not depend on the choice of the Morse function $f$, the Riemannian metric $g$ and the neighborhood $B_{2\delta}$.

The invariance with respect to $\delta$ is obvious. In fact, assume $f_\alpha\in\mathcal{F}_{Morse}^{\delta_1}, f_\beta\in\mathcal{F}_{Morse}^{\delta_2}$ and let $\delta=\max\{\delta_1,\delta_2\}$, then we have $f_\alpha, f_\beta \in \mathcal{F}_{Morse}^{\delta}$. Moreover, the local Morse homology even do not depend on the shape of the neighborhood of the critical point.

\begin{theorem}\label{thm of homology iso}
If $(f_\alpha, g_\alpha)$ and $(f_\beta, g_\beta)$ are two Morse pairs, then there exists isomorphism
$$\Phi_{\beta\alpha}: HM_*^{loc}(f_\alpha, g_\alpha)\rightarrow HM_*^{loc}(f_\beta, g_\beta).$$
\end{theorem}

\begin{corollary}
Betti numbers only depend on the manifold $M$ and the critical point $x_0$ and not Morse-Smale pair, i.e.$$\text{dim }HM_k^{loc}(f_\alpha, g_\alpha)=\text{dim }HM_k^{loc}(f_\beta, g_\beta).$$
\end{corollary}

Now we consider gradient flow lines for time dependent functions and metrics. Inspired by \cite{Frauenfelder}, we construct the time dependent function as following.
\begin{equation}\label{fs}
f_s(x(s))=\left\{
\begin{aligned}
&f_\alpha(x(s)), s\le -T, T>0,\\
&f_\beta(x(s)), s\ge T, \\
&\gamma(s) f_\alpha(x(s))+(1-\gamma(s))f_\beta(x(s)), -T<s<T.
\end{aligned}
\right.
\end{equation}
where $\delta\in \mathbb{R}^+$ and
$$\gamma(s)=\left\{
\begin{aligned}
&\frac{1}{2}e^\frac{1}{\delta}e^{-\frac{T}{(t+T)\delta}}, -T<s\le 0\\
&1-\frac{1}{2}e^\frac{1}{\delta}e^{-\frac{T}{(-t+T)\delta}}, 0<s< T.
\end{aligned}
\right.$$
Computing directly, we know that $\gamma'(s)$ exits for all $t\in [-T, T]$ and $\partial_s f_s(x(s))$ exists for all $t\in\mathbb{R}$.
In order to prove the local behavior of the time dependent gradient flow lines, we first estimate $\gamma'(s)$.
\begin{lemma}\label{lemma 2.5.1}
There exists $\delta>0$, such that
$$\gamma'(s)\le\frac{1}{T}.$$
\end{lemma}

\begin{proof}
$$\gamma'(s)=\frac{1}{2}e^{\frac{1}{\delta}}\frac{T}{(t+T)^2\delta}e^{-\frac{T}{(t+T)\delta}}.$$
$$\gamma''(s)=\frac{1}{2}e^{\frac{1}{\delta}\frac{T}{(t+T)^4\delta^2}}e^{\frac{T}{(t+T)\delta}}(T-2(t+T)\delta).$$

If $\delta\le \frac{1}{2}$,
$\gamma'(s)\ge 0.$
$\gamma'(s)$ gets the maximum value $h(\delta)=\frac{1}{2\delta T}$ at $t=0$.
and $h(\delta)$ gets the minimum value $\frac{1}{T}$ at $\delta=\frac{1}{2}$.

If $\delta\ge\frac{1}{2}$,
$\gamma'(s)$ gets the maximum value $h(\delta)=\frac{2\delta}{T} e^{\frac{1}{\delta}-2}$ at $t+T=\frac{T}{2\delta}$  ,
Since $$h'(\delta)=4e^{\frac{1}{\delta}-2}\frac{\delta-1}{\delta}<0,$$
$h(\delta)$ gets the minimum value $\frac{1}{T}$ at $\delta=\frac{1}{2}$.

Therefore, let $\delta=\frac{1}{2}$, we get $\gamma'(s)\le \frac{1}{T}$.
\end{proof}

Let
\begin{equation}\label{gs}
g_s(x(s))=\left\{
\begin{aligned}
&g_\alpha(x(s)), s\le -T, \\
&g_\beta(x(s)), s\ge T, T>0.
\end{aligned}
\right.
\end{equation}
be a smooth family of Riemannian metrics.

The gradient flow line $x(s): \mathbb{R}\rightarrow M$ is a solution of the time dependent O.D.E.
\begin{equation}\label{ODE(s)}
\partial_s x(s)+\nabla_{g_s}f_s\big(x(s)\big)=0, s\in\mathbb{R}.
\end{equation}

From the definition (\ref{fs}), (\ref{gs}) and (\ref{ODE(s)}), we can find that when $s\le -T$, $x(s)$ is the solution of
\begin{equation}\label{ODE1}
\partial_s x(s)+\nabla_{g_\alpha}f_\alpha=0, s\le -T
\end{equation}
 when $s\ge T$, $x(s)$ is the solution of
 \begin{equation}\label{ODE2}
 \partial_s x(s)+\nabla_{g_\beta}f_\beta=0, s\ge T.
 \end{equation}

 \begin{lemma}
There exists $\delta_0>0$, such that if $x(s)$ is a solution of (\ref{ODE(s)}), (\ref{ODE1}), or (\ref{ODE2}), then $x(s)\in B_{2\delta_0}$.
\end{lemma}

\begin{proposition}\label{weak compactness(s)}(Weak compactness of time dependent flow lines)
Let $x_\nu\in C^{\infty}(\mathbb{R}, M), \nu\in \mathbb{N}$ be solutions of equation (\ref{ODE(s)}), then there exist a subsequence $\nu_j$ and gradient flow line $x$, such that $$x_{\nu_j}\overset{C_{loc}^\infty}{\longrightarrow}x.$$ i.e. for any $R>0$,
$$x_{\nu_j}|_{[-R, R]}\overset{C^\infty}{\longrightarrow}x|_{[-R, R]}.$$
\end{proposition}
\begin{proof}
Step 1: The sequence $x_\nu$ is equicontinuous.
By (\ref{ODE(s)}),$$\nabla f_s\big(x(s)\big)=\frac{T-s}{2T}\nabla f_\alpha+\frac{T+s}{2T}\nabla f_\beta+\frac{1}{2T}f_\beta-\frac{1}{2T}f_\alpha.$$
 Since $f_\alpha$
\end{proof}

For $c_1\in\text{crit}f_\alpha, c_2\in\text{crit}f_\beta$, define linear map
\begin{equation}\label{phi}
\phi^{\beta\alpha}: CM_*^{loc}(f_\alpha)\rightarrow CM_*^{loc}(f_\beta),
\end{equation}
\begin{equation}
\begin{aligned}
&\phi^{\beta\alpha}(c_1)\\=&\sum_{\mu(c_1)=\mu(c_2)}\#_2 \{ x(s) \big| \partial_s x(s)+\nabla_{g_s}f_s=0, \lim_{s\rightarrow-\infty}x(s)=c_1, \lim_{s\rightarrow+\infty} x(s)=c_2\} c_2.
\end{aligned}
\end{equation}

Note that $$\mathcal{F}=\{ f\in L^{\infty}(M, \mathbb{R})|  f| _{M\backslash B_\delta}=\tilde f|_{M\backslash B_\delta}\},$$
$$\mathcal{F}_{Morse}=\{f\in \mathcal{F}: f|_{B_\delta} \text{ is a Morse function}\},$$
We know that $\mathcal{F}_{Morse}$ is dense in $\mathcal{F}$.
thus there exists a series of $f_\nu \in \mathcal{F}_{Morse}, \nu\in \mathbb{N}^*$, \begin{equation}\label{f lim}
f_\nu\overset{C^{\infty}}{\longrightarrow} \tilde f.
\end{equation}
In order to prove the time dependent flow lines behave locally, we should prove the following lemma. 
\begin{lemma}\label{lemma 2.5.2}
For any $\eta>0$, there exists $\nu_0$, such that for all $\nu_1,\nu_2>\nu_0$, we have $\max f_{\nu_1}|_{C_{\nu_1}}-\min f_{\nu_2}|_{C_{\nu_2}}<\eta$ and $|f_{\nu_1}(x)-f_{\nu_2}(x)|<\eta, \forall x\in B_\delta$, where $C_{\nu_i}=\{x\in B_\delta|\nabla f_{\nu_i (x)}=0\},i=1,2$.
\end{lemma}
\begin{proof}
We assume the contradiction of the lemma, then there exists $\eta_0$, for all $\nu_0$, there exists $\nu_1,\nu_2\in \mathbb{N}$, such that $\max f_{\nu_1}|_{C_{\nu_1}}-\min f_{\nu_2}|_{C_{\nu_2}}\ge\eta_0$ or there exits $\hat{x}\in B_\delta$ such that $|f_{\nu_1}(x_{\nu_1\nu_2})-f_{\nu_2}(x_{\nu_1\nu_2})|\ge\eta$.

Assume $\max f_{\nu_1}|_{C_{\nu_1}}-\min f_{\nu_2}|_{C_{\nu_2}}>\eta_0$ is true. Since $\#C_{\nu_i}<\infty,i=1,2$, there exist $x_{\nu_1}^1\in C_{\nu_1}, x_{\nu_2}^2\in C_{\nu_2}, \nu_1,\nu_2 \in \mathbb{N}^*$, such that $f_{\nu_1} (x_{\nu_1}^1)-f_{\nu_2} (x_{\nu_2}^2)\ge\eta_0$. Since $x_{\nu_1}^1, x_{\nu_2}^2$ are in the bounded regin $B_\delta$, there exist convergent subsequence $x^1_{\nu_{1_j}}\rightarrow x^1, x^2_{\nu_{2_k}}\rightarrow x^2, j,k\in \mathbb{N}^*$. As a result, $\tilde f(x^1)-\tilde f(x^2)\ge\eta_0$, this implies $\tilde f(x^1)\neq \tilde f(x^2)$ . By (\ref{f lim}), we know that $x^1$ and $x^2$ are critical points of $\tilde f$, i.e. $\tilde f(x_1)=\tilde f(x_2)=0$. This contradict the fact that $\tilde f$ has an unique critical point in $B_\delta$.

So there exits $\hat{x}\in B_\delta$ such that $|f_{\nu_1}(x_{\nu_1\nu_2})-f_{\nu_2}(x_{\nu_1\nu_2})|\ge\eta$. This implies that there exists a convergent subsequence $x_{{\nu_1\nu_2}_j}\rightarrow x, j\in\mathbb{N}^*$, thus we get an impossible result $0=|\tilde{f}(x)-\tilde{f}(x)|\ge \eta$. Therefore the conclusion of the lemma holds.
\end{proof}

Similar as in lemma \ref{lemma 1.2}, for time dependent gradient flow lines, we have the following lemma.

\begin{lemma}
If $x(s)$ is the gradient flow line of the time dependent O.D.E
\begin{equation}\label{ODE(s)}
\partial_s x(s)+\nabla_{g_s}f_s\big(x(s)\big)=0, s\in\mathbb{R}.
\end{equation}
then $x(s)\in B_{2\delta}$.
\end{lemma}
\begin{proof}
Assume the contradiction.
Firstly, similar as in the proof of Lemma \ref{lemma 1.2}, there exists $\epsilon_0>0$ such that $||\nabla f_s|_{A_\delta}||\ge \epsilon_0$.
As s result, we can prove that
\begin{equation}\label{E}
E(x)\geq \epsilon_0\delta.
\end{equation}

Secondly,
\begin{equation}
\begin{aligned}
E(x)&=\int^{\infty}_{-\infty}||\partial_s x||_{g_s}ds\\
&=-\int^{\infty}_{-\infty}g_s(\nabla_{g_s} f_s(x),\partial_s x)ds\\
&=-\int^{\infty}_{-\infty}df_s(x)\partial_s x ds\\
&=-\int^{\infty}_{-\infty}\frac{d}{ds}(f_s(x(s)))ds+\int^{\infty}_{-\infty}(\partial_s f_s)(x(s))ds\\
&=f_\alpha(c_1)-f_\beta(c_2)+\int^{T}_{-T}(\partial_s f_s)(x(s))ds\\
&=f_\alpha(c_1)-f_\beta(c_2)+\int^{T}_{-T}\gamma'(s)(f_\beta(x(s))-f_\alpha(x(s)))ds\\
\end{aligned}
\end{equation}

By Lemma $\ref{lemma 2.5.1}$ and Lemma $\ref{lemma 2.5.2}$, we get
\begin{equation}
\begin{aligned}
E(x)<3\eta.
\end{aligned}
\end{equation}
We fix $\eta=\frac{1}{3}\epsilon_0\delta$ in Lemma \ref{lemma 2.5.2}, and get $$E(x)<\epsilon_0\delta.$$

This contradicts (\ref{E}) and we get the conclusion.

\end{proof}

Now we begin to define the Floer-Gromov convergence for time dependent gradient flow lines.
\begin{definition}
A broken gradient flow line from $c_1$ to $c_2$ of $\nabla_{g_s}f_s$ is a tuple
$$y=\{x^k\}_{1\le k\le n}, n\in\mathbb{N},$$
such that

(i) $\exists k_0\in\{1,...,n\}$ such that $x^{k_0}$ is gradient flow line of $\nabla_{g_s} f_s$.

(ii) $x^k$ for $1\le k<k_0$ is a nonconstant gradient flow line of $\nabla_{g_\alpha} f_{\alpha}$,
$x^k$ for $k_0<k\le n$ is a nonconstant gradient flow line of $\nabla_{g_\beta}f_\beta$.

(iii) $\lim\limits_{s\rightarrow-\infty} x^1(s)=c_1, \lim\limits_{s\rightarrow\infty}x^k(s)=\lim\limits_{s\rightarrow-\infty}x^{k+1}(s), \lim\limits_{s\rightarrow\infty}x^n(s)=c_2, k\in\{1,...,n-1\}$
\end{definition}

By a proof similar to Proposition \ref{prop of weak compactness}, we have the following proposition.

\begin{proposition}(Weak compactness for time dependent gradient flow lines)
Let $x_\nu\in C^{\infty}(\mathbb{R}, B_{2\delta}), \nu\in \mathbb{N}$ be solutions of equation (\ref{ODE(s)}), then there exist a subsequence $\nu_j$ and a gradient flow line $x$, such that $$x_{\nu_j}\overset{C_{loc}^\infty}{\longrightarrow}x.$$ i.e. for any $R>0$,
$$x_{\nu_j}|_{[-R, R]}\overset{C^\infty}{\longrightarrow}x|_{[-R, R]}.$$
\end{proposition}

\begin{definition}(Floer-Gromov convergence for time dependent flow lines)
$x_\nu\in C^{\infty}(\mathbb{R}, B_{2\delta_0})$ is a sequence of gradient flow lines of $\nabla_{g_s} f_s$ satisfying $$\lim\limits_{s\rightarrow-\infty}x_\nu(s)=c_1, \lim\limits_{s\rightarrow+\infty}x_\nu=c_2$$ and $y=\{x^k\}_{1\le k\le n}$ is a broken flow line of $\nabla_{g_s}f_s$ from $c_1$ to $c_2$.

$x_v$ Floer-Gromov converges to $y$ if there exists a sequence $\Gamma_\nu^k\in\mathbb{R}$ for $1\le k\le n$ such that $$\Gamma_\nu^kx_\nu\overset{C_{loc}^\infty}{\longrightarrow}x^k.$$ 
\end{definition}

\begin{theorem}(Floer-Gromov compactness for time dependent gradient flow lines)
$x_\nu$ is a sequence of gradient flow lines of $\nabla_{g_s}f_s$ such that
$$\lim\limits_{s\rightarrow-\infty}x_\nu(s)=c_1, \lim\limits_{s\rightarrow+\infty}x_\nu=c_2.$$
Then there exist a subsequence $v_j$ and a broken gradient flow line $y=\{x^k\}_{1\le k\le n}$ from $c_1$ to $c_2$ such that $$x_{v_j}\xrightarrow{Floer-Gromov} y.$$
\end{theorem}

\begin{proof}
This is analogous to the proof of Theorem \ref{thm of Floer-Gromov compactness}, the time independent case. The only extra property needed is the uniform bound on energy we will prove in the next lemma.
\end{proof}



Define $\mathcal{N}(f_s, g_s; c_1, c_2)$ as the moduli space of gradient flow lines of $\nabla_{g_s}f_s$ from $c_1$ to $c_2$.

\begin{proposition}\label{N}
For generic homotopy, $\mathcal{N}(f_g, g_s; c_1, c_2)$ is a maninfold of
$$\text{dim} \mathcal{N}(f_s, g_s; c_1, c_2)=\mu(c_1)-\mu(c_2).$$

(i) If $\mu(c_1)=\mu(c_2)$, then $\mathcal{N}(f_s, g_s; c_1, c_2)$ is a finite set.

(ii) If $\mu(c_1)=\mu(c_2)+1$, then $\mathcal{N}$ can be compactified to a $1$-dim manifold with boundary $\overline{\mathcal{N}}(f_s, g_s; c_1, c_2)$ such that
\begin{equation}\label{DN}
\begin{aligned}
\partial \overline{\mathcal{N}}(f_s, g_s; c_1, c_2)=&\bigsqcup_{c\in \text{crit}(f_\alpha), \mu(c_1)-1=\mu(c)=\mu(c_2)}\mathcal{M}(f_\alpha, g_\alpha; c_1, c)\times\mathcal{N}(f_s, g_s; c, c_2)\\ &\sqcup \bigsqcup_{c\in\text{crit}(f_\beta),\mu(c_1)=\mu(c)=\mu(c_2)+1}\mathcal{N}(f_s, g_s; c_1, c)\times\mathcal{M}(f_\beta, g_\beta; c, c_2).
\end{aligned}\end{equation}
\end{proposition}
From (i) of Proposition \ref{N}, we know that the linear map (\ref{phi}) is well defined. From (\ref{DN}), we can get the equation
$$0=\phi^{\beta\alpha}\partial_\alpha+\partial_\beta\phi^{\beta\alpha}(\mod 2)$$
This equation is equivalent to the equation
\begin{equation}\label{chain map}
\phi^{\beta\alpha}\partial_\alpha=\partial_\beta\phi^{\beta\alpha}(\mod 2).
\end{equation}
i.e. $\phi^{\beta\alpha}$ is a chain map between chain complexes
   $(CM_*^{loc}(f_\alpha), \partial_\alpha)$ and
$(CM_*^{loc}(f_\beta),\partial_\beta)$.
\begin{proposition}
The linear map $\phi^{\beta\alpha}$ in (\ref{phi}) induces linear map
$$\Phi^{\beta\alpha}: HM_*^{loc}(f_\alpha, g_\alpha)\rightarrow HM_*^{loc}(f_\beta, g_\beta).$$
\end{proposition}
\begin{proof}

Choose $\xi\in CM_*(f_\alpha)$ such that $\ker \partial^\alpha\xi=0$, then by (\ref{chain map}) we have
$$\partial^\beta\phi^{\beta\alpha}\xi=\phi^{\beta\alpha}\partial^\alpha\xi=0,$$
so $\phi^{\beta\alpha}\xi$ lies in the kernal of $\partial^\beta$ and represents a homology class in $H_*(f_\beta, g_\beta)$.

Assume $\gamma\in \text{im}\partial_\alpha$, then there exists $\eta\in CM_*^{loc}(f_\alpha) $, such that $\partial_\alpha \eta=\gamma$. By (\ref{chain map}), we have
$$\phi^{\beta\alpha}(\gamma)=\phi^{\beta\alpha}(\partial_\alpha\eta)=\partial_\beta\phi^{\beta\alpha}\eta\in \text{im}\partial_\beta,$$
so the class $[\phi^{\beta\alpha}\xi]\in HM_*^{loc}(f_\beta, g_\beta)$ does not depend on the the choice of $\xi$ as a representation of its homology class $HM_*^{loc}(f_\alpha, g_\alpha)$.
\end{proof}

\begin{proposition}
$\Phi^{\beta\alpha}$ is independent of the homotopy $(f_s, g_s)$.
\end{proposition}
\begin{proof}
Assume $(f_s^0, g_s^0)$ and $(f_s^1, g_s^1)$ are two homotopies from $(f_\alpha, g_\alpha)$ to $(f_\beta, g_\beta)$. Consider homotopy of these two homotopies
$$(f_s^r, g_s^r), r\in[0,1].$$
Let $c_1\in \text{crit} f_\alpha$, $c_2\in \text{crit} f_\beta$, we introduce the moduli space
$$\begin{aligned}
&\mathcal{R}(f_s^r, g_s^r; c_1, c_2)\\
=&\{(x,r)| r\in[0,1], \partial_s x+\nabla_{g_s^r}f_s^r(x)=0, \lim_{s\rightarrow-\infty}x(s)=c_1,\lim_{s\rightarrow\infty} x(s)=c_2.\}\end{aligned}$$
$$\text{dim}\mathcal{N}(f_s^r,g_s^r; c_1, c_2)=\mu(c_1)-\mu(c_2)+1$$
Assume $\mu(c_1)=\mu(c_2)$, then $\text{dim} \mathcal{R}(f_s^r, g_s^r; c_1, c_2)=1$.
For generic homotopies of homotopies
$\mathcal{R}(f_s^r, g_s^r; c_1, c_2)$ can be compactified to compact $1$-dim manifold with boundary $\bar{\mathcal{R}}(f_s^r, g_s^r; c_1, c_2)$.

There exist finite $R\in[0,1]$, such that
\begin{equation}\label{partial R}
\begin{aligned}
&\partial\bar{\mathcal{R}}(f_s^r, g_s^r; c_1, c_2)\\
=&\mathcal{R}(f_s^0,g_s^0; c_1, c_2)\sqcup\mathcal{R}(f_s^1,g_s^1;c_1,c_2)\\
&\sqcup\bigsqcup_{r_0\in R, c\in\text{crit}_{k+1}f_\alpha,\mu(c)=\mu(c_1)-1=\mu(c_2)-1}\mathcal{R}(f_\alpha, g_\alpha; c_1,c)\times\mathcal{N}(f_s^{r_0},g_s^{r_0};c,c_2)\\
&\sqcup\bigsqcup_{r_1\in R,c\in\text{crit}_{k-1}f_\beta,\mu(c)=\mu(c_1)+1=\mu(c_2)+1}\mathcal{N}(f_s^{r_1},g_s^{r_1};c_1,c)\times\mathcal{R}(f_\beta,g_\beta; c, c_2).
\end{aligned}
\end{equation}

For $c_1\in \text{crit}(f_\alpha)$, define linear map $T: CM_*^{loc}(f_\alpha)\rightarrow CM_{*+1}^{loc}(f_\beta)$,
$$T(c_1)=\sum_{c_2\in\text{crit}(f_\beta),\mu(c_2)=\mu(c_1)+1}\#_2(\bigsqcup_{r_0\in R}\mathcal{N}(f_s^{r_0},g_s^{r_0}; c_1, c_2))c_2.$$
From (\ref{partial R}), we get

$$0=\phi^{\beta\alpha}_{(f_s^0,g_s^0)}+\phi^{\beta\alpha}_{(f_s^1,g_s^1)}+\partial_\beta T+T\partial_\alpha,$$
and then
$$\phi_0^{\beta\alpha}=\phi_1^{\beta\alpha}+\partial_\beta T+T\partial_\alpha.$$
i.e. $\phi_0^{\beta\alpha}$ is chain homotopy equivalent to $\phi_1^{\beta\alpha}$.
As a result, $\Phi_0^{\beta\alpha}=\Phi_1^{\beta\alpha}$.
\end{proof}

\begin{proposition}\label{functionality}
$(f_\alpha, g_\alpha), (f_\beta, g_\beta), (f_\gamma, g_\gamma)$ are Morse-Smale pairs, then $\Phi^{\gamma\alpha}= \Phi^{\gamma\beta}\circ\Phi^{\beta\alpha}: HM_*^{loc}(f_\alpha, g_\alpha)\rightarrow HM_*^{loc}(f_\gamma, g_\gamma)$.
\end{proposition}
\begin{proof}
Choose a generic homotopy $(f_s^{\beta\alpha}, g_s^{\beta\alpha})$ between $(f_\alpha, g_\alpha)$ and $(f_\beta, g_\beta)$,  $(f_s^{\gamma\beta}, g_s^{\gamma\beta})$ between $(f_\beta, g_\beta)$  and $(f_\gamma, g_\gamma)$.
There exists $T>0$, such that
$$f_s^{\beta\alpha}=\left\{\begin{matrix}
f_\alpha, s\le -T,\\
f_\beta, s\ge T.
\end{matrix}\right.$$
$$f_s^{\gamma\beta}\left\{
\begin{matrix}
f_\beta, s\le -T,\\
f_\gamma, s\ge T.
\end{matrix}
\right.$$
$$\begin{aligned}&\Phi^{\gamma\beta}\circ\Phi^{\beta\alpha}(x)\\=&\sum_{z\in\text{crit}f_\gamma,\mu(z)=\mu(y)}\sum_{y\in\text{crit}f_\beta,\mu(y)=\mu(x)}\#_2
(\mathcal{N}(f^{\beta\alpha}_s,g_s^{\beta\alpha}; x,y)\times(\mathcal{N}(f_s^{\gamma\beta},g_s^{\gamma\beta};y,z))z.\end{aligned}$$
Let $u\in \mathcal{N}(f^{\beta\alpha}_s,g_s^{\beta\alpha}; x,y)$, $v\in \mathcal{N}(f_s^{\gamma\beta},g_s^{\gamma\beta};y,z)$, now we glue $u$ and $v$ together.
$$u\#_R v: \mathbb{R}\rightarrow M,$$
$$u\#_R v(s)=\left\{
\begin{aligned}
&u(s+T+R), s\le -R,\\
&u(\zeta_1(s)), -R\le s\le 0,\\
&y, s=0,\\
&u(\zeta_2(s)), 0\le s\le R,\\
&v(s-T-R), s\ge R.
\end{aligned}
\right.$$
Where $$\zeta_1: [-R,0)\rightarrow [T,+\infty)$$
and $$\zeta_2: (0,R]\rightarrow(-\infty,-T]$$
are smooth monotonic bijections.

Let $$f_s^{\gamma\alpha, R}=\left\{
\begin{aligned}
&f^{\beta\alpha}_{s+(T+R)}, s\le-R,\\
&f_\beta, -R\le s\le R,\\
&f^{\gamma\beta}_{s-(T+R)}, s\ge R.
\end{aligned}
\right.$$
Then $u\#_R v$ is a gradient flow line of
$\nabla_{g_s}^{\gamma\alpha,R}f_s^{\gamma\alpha,R}$ for $s\le -R$ and $s\ge R$ and almost gradient flow line for $s\in[-R, R]$. Using implicit function theorem we can get an actual gradient flow line close by .
Therefore $$\phi^{\gamma\beta}_{(f_s^{\gamma\beta},g_s^{\gamma\beta})}\circ\phi^{\beta\gamma}_{(f_s^{\beta\gamma},g_s^{\beta\alpha})}=\phi^{\gamma\alpha}_{(f_s^{\gamma\alpha,R},g_s^{\gamma\alpha,R})},$$
and
$$\Phi^{\gamma\beta}\circ\Phi^{\beta\alpha}=\Phi^{\gamma\alpha}: HM_*^{loc}(f_\alpha,g_\alpha)\rightarrow HM_*^{loc}(f_\gamma,g_
\gamma).$$

\end{proof}

\begin{proposition}
If $\Phi^{\alpha\alpha}: HM_*^{loc}(f_\alpha,g_\alpha)\rightarrow HM_*^{loc}(f_\alpha, g_\alpha)$, then $\Phi^{\alpha\alpha}=id$.
\end{proposition}
\begin{proof}
Choose a constant homotopy
$$\begin{aligned}&f_s=f_\alpha,\\
&g_s=g_\alpha, s\in \mathbb{R}.
\end{aligned}$$
Then
$$\phi^{\alpha\alpha}(x): CM_*^{loc}(f_\alpha)\rightarrow CM_*^{loc}(f_\alpha).$$
$$\phi^{\alpha\alpha}(x)=\sum_{y\in\text{crit}f_\alpha,\mu(y)=\mu(x)}\#_2\mathcal{N}(f_\alpha, g_\alpha; x, y)y.$$
Where
$$\mathcal{N}(f_\alpha,g_\alpha; x, y)=\tilde{\mathcal{M}}(f_\alpha, f_\alpha; x, y).$$
and $$\mathcal{M}(f_\alpha, f_\alpha; x, y)=\tilde{\mathcal{M}}(f_\alpha, f_\alpha; x, y)/\mathbb{R}$$
If $y\neq x$, then
$$\dim \mathcal{M}(f_\alpha, g_\alpha; x, y)=\mu(x)-\mu(y)-1=-1,$$
thus $$\mathcal{M}(f_\alpha, f_\alpha; x, y)=\emptyset$$ and $$\tilde{\mathcal{M}}(f_\alpha, f_\alpha; x, y)=\emptyset.$$
If $y=x$, then$$\tilde{\mathcal{M}}(f_\alpha, f_\alpha; x, y)=\{x\}.$$
Therefore, $$\#_2\tilde{M}(f_\alpha, g_\alpha; x, y)=\left\{\begin{matrix}0,y\neq x,\\1,y=x.\end{matrix}\right.$$
As a result, $$\phi^{\alpha\alpha}=id: CM_*^{loc}(f_\alpha)\rightarrow CM_*^{loc}(f_\alpha),$$
and $$\Phi^{\alpha\alpha}=id: HM_*^{loc}(f_\alpha, g_\alpha)\rightarrow HM_*^{loc}(f_\alpha, g_\alpha).$$
\end{proof}

\begin{proposition}
$\Phi^{\beta\alpha}: HM_*^{loc}(f_\alpha, g_\alpha)\rightarrow HM_*^{loc}(f_\beta, g_\beta)$ is an isomorphism.
\end{proposition}
\begin{proof}
By proposition \ref{functionality}, we get
$$\Phi^{\alpha\beta}\circ\Phi^{\beta\alpha}=\Phi^{\alpha\alpha}=id: HM_*^{loc}(f_\alpha, g_\alpha)\rightarrow HM_*^{loc}(f_\alpha, g_\alpha).$$
$$\Phi^{\beta\alpha}\circ\Phi^{\alpha\beta}=\Phi^{\beta\beta}=id: HM_*^{loc}(f_\beta, g_\beta)\rightarrow HM_*^{loc}(f_\beta, g_\beta).$$
Therefore $$\Phi^{\alpha\beta}: HM_*^{loc}(f_\alpha, g_\alpha)\rightarrow HM_*^{loc}(f_\beta, g_\beta)$$ is an isomorphism with inverse
$$(\Phi^{\alpha\beta})^{-1}=\Phi^{\beta\alpha}.$$
\end{proof}

\begin{theorem}\label{Homotopy invarance}
Local Morse Homology is invariant under a small perturbation of the smooth function if all the critical points of the perturbed smooth function are still isolated.
\end{theorem}

\begin{proof}

Assume $x_1$ is a critical point of the smooth function $\tilde f_1$, i.e.
$$\nabla \tilde f_1(x_1)=0.$$
Note that $$\Delta(\tilde f,x_0)=\{\delta| x_0 \text{ is the unique critical point of } \tilde f \text{ in } B_{2\delta}\}.$$
We choose $\delta_1\in \Delta(\tilde f_1,x_1)$ and perturb $\tilde f_1$ small enough to get a smooth function $\tilde f_2$ with an isolate critical point $x_2$ near $x_1$ such that $x_2\in B(x_1,2\delta)$. Let $\delta=\max\{\delta_1,\delta_2\}$ and $$N=B(x_1,2\delta)\cap B(x_2,2\delta).$$
Let $(f_1,g_1)$ and $(f_2,g_2)$ be Morse-Smale pairs of $\tilde f_1$ and $\tilde f_2$ respectively.
From the proof of Lemma \ref{lemma 1.2}, we know that local Morse homology only depends on the properties of the Morse function and the neighborhood, does not depend on the isolated critical point of the smooth function, namely, it does not depend on the choice of $x_1$ and $x_2$ here. By Theorem \ref{thm of homology iso}, local Morse homology on N does not depend on the choice of Morse-Smale pairs and the shape of the neighborhood. Therefore, $HM^{loc}(\tilde f_1,g_1,x_1,\delta)=HM^{loc}(\tilde f_2,g_2,x_2,\delta)$.
\end{proof}

By the compactness of the homotopy and Heine-Borel theorem we can get the following theorem directly.
\begin{theorem}(Homotopy invariance of local Morse homology)\label{Homotopy invariance}
Local Morse homology is invariant under a homotopy of smooth functions if the critical points of the smooth functions are always isolated under the homotopy.
\end{theorem}

\section{Applications in the Lagrange problem}
The Lagrange problem is the problem of two fixed centers  adding an elastic force from the middle of the two fixed centers. Setting the two fixed centers as $e=(-\frac{1}{2},0)$ and $m=(\frac{1}{2},0)$, its Hamiltonian function is

\begin{equation}H(q,p)=T(p)-U(q).\end{equation}\label{H of Lagrange}
where
$$T(p)=\frac{1}{2}|p|^2.$$
$$U:\mathbb{R}^2\backslash\{(\pm\frac{1}{2},0)\}\rightarrow \mathbb{R},
q\rightarrow \frac{m_1}{\sqrt{(q_1+\frac{1}{2})^2+q_2^2}}+\frac{m_2}{\sqrt{(q_1-\frac{1}{2})^2+q_2^2}}+\frac{\epsilon}{2}|q|^2,$$
$p=(p_1,p_2)^T, q=(q_1,q_2)^T$
and $m_1,m_2,\epsilon\in \mathbb{R}^+$.

By \cite{Tang}, the Lagrange problem has five critical points $l_i, i=1, 2, 3,4,5$, where $l_4, l_5$ are maxima by direct computations.
Now we give the local Morse homology of these critical points. 


\begin{theorem}\label{homology of Lagrange}
For the collinear critical points $l_i(i=1, 2, 3)$ of the Lagrange problem with $m_1,m_2,\epsilon\in \mathbb{R}^+$, we have
$$HM^{loc}_*(l_i)=\left\{
\begin{aligned}
&\mathbb{Z}_2, *=1;\\
&\{0\}, otherwise.
\end{aligned}
\right.$$
For the maxima $l_i(i=4,5)$, we have
$$HM^{loc}_*(l_i)=\left\{
\begin{aligned}
&\mathbb{Z}_2, *=2;\\
&\{0\}, otherwise.
\end{aligned}
\right.$$
\end{theorem}

\begin{proof}
Since $l_i(i=4,5)$ are maxima of the Lagrange problem by direct computation in \cite{Tang}, we get the their local Morse homology directly.
Firstly, we consider the Lagrange problem with $m_1=m_2$ and $m> \frac{\epsilon}{16}$. Its Hamiltonian is
\begin{equation}\label{r3bp}
H_r(q,p)=\frac{1}{2}|p|^2
-\frac{m}{\sqrt{(q_1+\frac{1}{2})^2+q_2^2}}-\frac{m}{\sqrt{(q_1-\frac{1}{2})^2+q_2^2}}-\frac{\epsilon}{2}|q|^2,
\end{equation}
for
$$
V_r(q,p)=-\frac{m}{\sqrt{(q_1+\frac{1}{2})^2+q_2^2}}-\frac{m}{\sqrt{(q_1-\frac{1}{2})^2+q_2^2}}-\frac{\epsilon}{2}|q|^2.$$
Let $l_i^r, i=1, 2, 3$ be the critical points of $U_r$ on the $x$-axis. we know that $l_1=(0,0)$ because of the symmetric property of $U_r$.
Because $U$ is invariant under reflection at the $q_1$-axis and the three collinear critical points are fixed points of this flection, we conclude that
$$\frac{\partial^2V}{\partial q_1q_2}(l_i)=0, i=1,2,3.$$
By a direct computation, we get
\begin{equation}\label{partial U q1q1 on x-axis}
\frac{\partial^2 V_r}{\partial q_1^2}\bigg|_{q_2=0}=-\frac{2m_1}{|q_1+\frac{1}{2}|^3}-\frac{2m_2}{|q_1-\frac{1}{2}|^3}-\epsilon<0,
\end{equation}
and
\begin{equation}
\begin{aligned}
\frac{\partial^2V_r}{\partial q_2^2}\bigg |_{q=l_1}=&\frac{m}{|q_1+\frac{1}{2}|^3}+\frac{m}{|q_1-\frac{1}{2}|^3}-\epsilon=16m-\epsilon >0,\\
\frac{\partial^2V_r}{\partial q_2^2}\bigg |_{q=l_2}=&\frac{m\big((q_1+\frac{1}{2})^3-(q_1-\frac{1}{2})^3\big)}{2q_1(q_1+\frac{1}{2})(q_1-\frac{1}{2})}>0,\\
\frac{\partial^2V_r}{\partial q_2^2}\bigg |_{q=l_3}>&0.
\end{aligned}
\end{equation}
As a result, $l_1,l_2,l_3$ are saddle points of $V_r$, and we have
$$HM^{loc}_*(l_i^r)=\left\{
\begin{aligned}
&\mathbb{Z}_2, *=1;\\
&\{0\}, otherwise.
\end{aligned}
\right.$$

We construct a homotopy from the potential energy $V_r(q,p)$ of the Lagrange problem with equal mass $m$ satisfying $m\ge \epsilon$ to any other Lagrange problem with potential energy $U(q,p)$.
$$U^t(q,p)=-\frac{m+t(m_1-m)}{\sqrt{(q_1+\frac{1}{2})^2+q_2^2}}-\frac{m+t(m_2-m)}{\sqrt{(q_1+\frac{1}{2})^2+q_2^2}}-\frac{\epsilon}{2}|q|^2, t\in [0,1].$$

From \cite{Tang}, we know that $U_r$ always has five isolated critical points under the homotopy $U^t(q,p)$, then by Lemma \ref{Homotopy invarance}, the local Morse homology does not change. Therefore, we get the conclusion.
\end{proof}
From Lemma \ref{homology of Lagrange}, we get the following conclusion directly.
\begin{corollary}\label{saddle}
Each of the three critical points of the Lagrange problem with $m_1,m_2,\epsilon\in\mathbb{R}^+$ on the x-axis is either a saddle point or a degenerate critical point.
\end{corollary}

\begin{remark}
In \cite{Tang} , for the Lagrange problem, using Euler characteristic and the Poincar\'e-Hopf index theory, we can only prove that when the critical points on the x-axis are all non-degenerate, they are saddle points. In this paper, using the local Morse homology we construct, we can show it can happen that some of the three critical points are degenerate and the rest of them are saddle points.
\end{remark}

\section{Acknowledgments}

I would like to express my sincere gratitude to Urs Frauenfelder for his invaluable guidance and insightful suggestions throughout this research.

\end{document}